\documentclass[11pt,reqno]{amsart}
\usepackage{graphicx}
\usepackage[caption = false]{subfig}
\usepackage{amsmath,amsfonts,amssymb,comment}
\usepackage{mathrsfs}
\usepackage{siunitx}
\usepackage{float} 
\usepackage{enumerate}
\usepackage{enumitem}
\usepackage{mathtools}
\usepackage{amsthm}
\usepackage{varioref}
\usepackage[english]{babel}
\usepackage[utf8x]{inputenc}\usepackage[english]{babel}
\usepackage{soul}
\usepackage{enumitem}
\usepackage{xcolor}

\usepackage{hyperref} 

\usepackage{cite}
\usepackage{relsize} 

\usepackage{titlesec}
\usepackage[margin=1in]{geometry} 
\usepackage{setspace} 

\titleformat{\section}[block]
  {\normalfont\Large\bfseries}{\thesection}{1em}{}
\titlespacing*{\section}
  {0pt}{10pt plus 4pt minus 2pt}{7pt plus 4pt minus 2pt}

\titleformat{\subsection}[block]
  {\normalfont\large\bfseries}{\thesubsection}{1em}{}
\titlespacing*{\subsection}
  {0pt}{8pt plus 3pt minus 2pt}{5pt plus 2pt minus 1pt}

\newtheorem{theorem}{Theorem}[section]

\newtheorem{lemma}[theorem]{Lemma}

\theoremstyle{definition}

\theoremstyle{remark}

\numberwithin{equation}{section}

\begin{document}
	\title[On Coefficient problems for \textbf{$S^*_{\rho}$}]{On Coefficient problems for \textbf{$S^*_{\rho}$}}
	
	    \author[S. S. Kumar]{S. Sivaprasad Kumar$^*$}
	\address{Department of Applied Mathematics, Delhi Technological University, Bawana Road, Delhi-110042, INDIA}
	\email{spkumar@dce.ac.in}
        \author[A. Tripathi]{Arya Tripathi}
	\address{Department of Applied Mathematics, Delhi Technological University, Bawana Road, Delhi-110042, INDIA}
	\email{aryatripathiiii@gmail.com}
	\author[S. Pannu]{Snehal Pannu}
	\address{Department of Applied Mathematics, Delhi Technological University, Bawana Road, Delhi-110042, INDIA}
	\email{s02nehal@gmail.com}

	
	
\let\thefootnote\relax\footnotetext{*Corresponding author}
\begin{abstract}
Logarithmic and inverse logarithmic coefficients play a crucial role in the theory of univalent functions. In this study, we focus on the class of starlike functions \(\mathcal{S}^*_\rho\), defined as  
\[
\mathcal{S}^*_\rho = \left\{ f \in \mathcal{A} : \frac{z f'(z)}{f(z)} \prec \rho(z), \; z \in \mathbb{D} \right\},
\]  
where \(\rho(z) := 1 + \sinh^{-1}(z)\), which maps the unit disk \(\mathbb{D}\) onto a petal-shaped domain.  
This investigation aims to establish bounds for the second Hankel and Toeplitz determinants, with their entries determined by the logarithmic coefficients of \(f\) and its inverse \(f^{-1}\), for functions \(f \in \mathcal{S}^*_\rho\).

\end{abstract}

\maketitle


\section{\hspace{5pt} Introduction}
\label{intro}
 Let \( \mathcal{A} \) denote the class of all normalized analytic functions \( f \) defined on the open unit disk
$
\mathbb{D} := \{z \in \mathbb{C} : |z| < 1\},
$
such that \( f(0) = 0 \) and \( f'(0) = 1 \). Each function \( f \in \mathcal{A} \) admits a Taylor series expansion of the form:  
\begin{equation}
    f(z) = z + \sum_{n=2}^{\infty} a_n z^n. \label{eq:1.1}
\end{equation}

Let $\mathcal{S} \subset \mathcal{A}$, where $\mathcal{S}$ represents the collection of  analytic univalent functions. In particular, the class $\mathcal{S}^*$, comprising starlike functions, maps the unit disk $\mathbb{D}$ onto a starlike domain, represented as:
\[
\mathcal{S}^* = \left\{ f \in \mathcal{A} : \operatorname{Re} \;\frac{z f'(z)}{f(z)} > 0, \quad z \in \mathbb{D} \right\}.
\]

Let $f$ and $g$ be two analytic functions, we say that $f$ is subordinate to $g$, denoted by $f \prec g$, if there exists a Schwarz function $w$ with $w(0) = 0$ and $|w(z)| < 1$ such that $f(z) = g(w(z))$. In 1992, Ma and Minda \cite{9} introduced a broader subclass of $\mathcal{S}^*$, denoted by $\mathcal{S}^*(\varphi)$, and defined as:
\[
\mathcal{S}^*(\varphi) = \left\{ f \in \mathcal{A} : \frac{z f'(z)}{f(z)} \prec \varphi(z) \right\},
\]
where $\varphi$ is an analytic univalent function satisfying the conditions $\operatorname{Re} \varphi(z) > 0$, with $\varphi(\mathbb{D})$ being starlike with respect to $\varphi(0) = 1$, $\varphi'(0) > 0$, and the domain $\varphi(\mathbb{D})$ being symmetric about the real axis.
Various subclasses of $\mathcal{S}^*$ arise when the choice of $\varphi$ is varied in the class $\mathcal{S}^*(\varphi)$.

Arora and Kumar \cite{3} introduced a class of starlike functions associated with the petal-shaped domain $\rho(\mathbb{D})$, by choosing $\varphi(z):= \rho(z)$,  where $\rho(z)=1 + \sinh^{-1}(z)$, defined by:

\begin{equation}
    \mathcal{S}^*_\rho = \left\{ f \in \mathcal{A} : \frac{z f'(z)}{f(z)} \prec 1 + \sinh^{-1}(z), \quad z \in \mathbb{D} \right\}.\nonumber
\end{equation} 

Clearly, the function $\rho(z)$ is defined over the branch cuts along the segments $(-\infty, -i) \cup (i, \infty)$ on the imaginary axis. Consequently, the function is analytic in the unit disk $\mathbb{D}$, geometrically, $\rho(\mathbb{D})=\Omega_\rho := \{ \Omega \in \mathbb{C} : |\sinh(\Omega - 1)| < 1 \}$. Certain coefficient problems of the analytic functions associated with the class $\mathcal{S}^*_\rho $ were discussed in \cite{7}. The function $f_{0}(z)$ serves as an extremal function for the family $\mathcal{S}^{*}_{\rho}$ which helps in obtaining the sharp results, given by
\begin{equation}
    f_{0}(z) = z \exp\left(\int_{0}^{z} \frac{\sinh^{-1}(t)}{t} \, dt\right) = z + z^2 + \frac{1}{2} z^3 + \frac{1}{9} z^4 - \frac{1}{72} z^5 - \frac{1}{225} z^6 + \ldots \;. \label{1.45}
\end{equation}
 
 The logarithmic coefficients $\gamma_n$ associated with $f \in \mathcal{S}$ are defined as follows:
\begin{equation}
F_f(z) = \log\left( \frac{f(z)}{z}\right) = 2 \sum_{n=1}^{\infty} \gamma_n(f) z^n; \quad z \in \mathbb{D} \setminus \{0\}, \quad \log 1 = 0. \nonumber
\end{equation}
We denote $\gamma_n(f)$ by $\gamma_n$. For functions in class $\mathcal{S}$, the sharp bounds for the logarithmic coefficients when $n = 1$ and $n = 2$ are:
$$|\gamma_1| \leq 1, \quad |\gamma_2| \leq \frac{1}{2} + \frac{1}{e^2}.$$

A significant motivation behind studying logarithmic coefficients lies in estimating the sharp bounds for the class $\mathcal{S}$. Currently, these sharp bounds have only been established for $\gamma_1$ and $\gamma_2$. The problem of determining the sharp bounds for $\gamma_n$ when $n \geq 3$ is still open. The Koebe $1/4$-theorem, enables us to define the inverse function $F_f \in \mathcal{A}$ in a neighborhood of the origin with the Taylor series expansion:
\begin{equation}
    F(w) := f^{-1}(w) = w + \sum_{n=2}^{\infty} A_n w^n, \quad |w| < 1. \label{eq:1.108}
\end{equation}

 The logarithmic inverse coefficients $\Gamma_n$, for $n \in \mathbb{N}$, associated with the function $F$, are defined by the relation
\[
F_{f^{-1}}(w):=\log\left(\frac{F(w)}{w}\right) = 2\sum_{n=1}^{\infty} \Gamma_n w^n, \quad |w| < \frac{1}{4}.
\]
Ponnusamy \emph{et al.} \cite{13} established sharp bounds for the initial logarithmic coefficients of the inverse function for certain subclasses of $\mathcal{S^*}$.

Hankel matrices (and determinants) have emerged as fundamental elements in different areas of mathematics, finding a wide array of applications \cite{17}.
For a function $f \in \mathcal{A}$, the $q^{th}$ Hankel determinant $H_{q,n}(F_f)$, where $q, n \in \mathbb{N}$, with entries as logarithmic coefficients, is given by:

\begin{equation}
    H_{q,n}(F_{f}/2) = \begin{vmatrix}
\gamma_n & \gamma_{n+1} & \dots & \gamma_{n+q-1}\\
\gamma_{n+1} & \gamma_{n+2} & \dots & \gamma_{n+q} \\
\vdots & \vdots & \ddots & \vdots\\
\gamma_{n+q-1} & \gamma_{n+q} & \dots & \gamma_{n+2(q-1)}\\
\end{vmatrix}. \label{1.10}
\end{equation}

Kowalczyk \emph{et~al.} \cite{1} investigated the Hankel determinants formed using logarithmic coefficients. Specifically, the expression $H_{2,1}(F_{f}/2)$ shares a notable resemblance to $H_{2,1}(f) = a_2 a_4 - a_3^2$, where $f \in \mathcal{A}$. Hankel determinants are beneficial, for example, in determining whether certain coefficient functionals related to functions are bounded in $\mathbb{D}$ and if they achieve sharp bounds, see \cite{16}. Furthermore, Mundalia and Kumar \cite{10} explored the problem of logarithmic coefficients for certain subclasses of close-to-convex functions. The logarithmic coefficients for a function $f \in \mathcal{S}$ are represented as follows:
\begin{equation}
    \gamma_1 = \frac{1}{2} a_2, \quad \gamma_2 = \frac{1}{2} \left(a_3 - \frac{1}{2} a_2^2 \right), \quad \gamma_3 = \frac{1}{2} \left( a_4 - a_2 a_3 + \frac{1}{3} a_2^3 \right). \label{log coeff} 
\end{equation}

In 2016, Ye and Lim \cite{17} established that any $n \times n$ matrix over $\mathbb{C}$ can be represented, in general, as a product of certain Toeplitz or Hankel matrices. Toeplitz matrices, along with their determinants, play a significant role in both applied and theoretical mathematics. Their applications extend to various fields such as analysis, quantum physics, image processing, integral equations, and signal processing. A defining feature of Toeplitz matrices is that their elements remain constant along each diagonal. Giri and Kumar \cite{6} explored bounds on Toeplitz determinants for different subclasses of normalized univalent functions in higher-dimensional spaces. Inspired by these contributions, we define the Toeplitz determinant of the logarithmic coefficients for $f \in \mathcal{S}$ as follows:

\begin{equation}
   T_{q,n}(F_{f}/2) = \begin{vmatrix}
\gamma_n & \gamma_{n+1} & \dots & \gamma_{n+q-1}\\
\gamma_{n+1} & \gamma_{n} & \dots & \gamma_{n+q-2} \\
\vdots & \vdots & \ddots & \vdots\\
\gamma_{n+q-1} & \gamma_{n+q-2} & \dots & \gamma_{n}\\
\end{vmatrix}. \label{1.12}
\end{equation} 

Motivated further by the above concepts, we begin the examination of the Hankel determinant $H_{q,n}(F_{f^{-1}}/2)$ and the Toeplitz determinant $T_{q,n}(F_{f^{-1}}/2)$, where the elements are the logarithmic coefficients of the inverse functions \cite{14} of $f^{-1} \in \mathcal{S}$. The Hankel determinant is given by
 
\begin{equation}
H_{q,n}(F_{f^{-1}}/2) = \begin{vmatrix}
\Gamma_n & \Gamma_{n+1} & \dots & \Gamma_{n+q-1}\\
\Gamma_{n+1} & \Gamma_{n+2} & \dots & \Gamma_{n+q} \\
\vdots & \vdots & \ddots & \vdots\\
\Gamma_{n+q-1} & \Gamma_{n+q} & \dots & \Gamma_{n+2(q-1)}\\
\end{vmatrix}, \label{1.13}
\end{equation}

and the Toeplitz determinant $T_{q,n}(F_{f^{-1}}/2)$ is given by
\begin{equation}
T_{q,n}(F_{f^{-1}}/2) = \begin{vmatrix}
\Gamma_n & \Gamma_{n+1} & \dots & \Gamma_{n+q-1}\\
\Gamma_{n+1} & \Gamma_{n} & \dots & \Gamma_{n+q-2} \\
\vdots & \vdots & \ddots & \vdots\\
\Gamma_{n+q-1} & \Gamma_{n+q-2} & \dots & \Gamma_{n}\\
\end{vmatrix}. \label{1.14}
\end{equation}
This work focuses on deriving bounds for the second Hankel and Toeplitz determinants, where the entries are the logarithmic coefficients associated with the functions $f$ and $f^{-1}$. Here, $f$ belongs to the class $\mathcal{S}^*_\rho$.

\section{\hspace{5pt} Preliminary results}\label{3}
The Carathéodory class $\mathcal{P}$ and its coefficient bounds are crucial in determining Hankel and Toeplitz determinant bounds. Let $\mathcal{P}$ denote the class of functions consisting of $p$, such that
\begin{equation}
    p(z) = 1 + \sum_{n=1}^\infty p_n z^n, \label{pp}
\end{equation}
which are analytic in the open unit disk $\mathbb{D}$ and satisfy $\operatorname{Re} \, p(z) > 0$ for any $z \in \mathbb{D}$. Here $p(z)$ is called the Carathéodory function, or functions with positive real part \cite{2, 5}. It is known that $p_n \leq 2$, $n \geq 1$, for a function $p \in \mathcal{P}$. In this section, we provide crucial lemmas that will be utilized to establish the main results of this paper.

\begin{lemma}\label{lem:1}
\cite{2,8,11}: If $p \in \mathcal{P}$ is the form \eqref{pp}, then 
\begin{equation}
    p_{1}=2\zeta_{1}, \label{l1}
\end{equation}
\begin{equation}
p_{2}=2\zeta_{1}^{2}+2(1-\zeta_{1}^{2})\zeta_{2}, \label{l2}
\end{equation}
\begin{equation}
p_{3}=2\zeta_{1}^{3}+4(1-\zeta_{1}^{2})\zeta_{1}\zeta_{2}-2(1-\zeta_{1}^{2})\zeta_{1}\zeta_{2}^{2}+2(1-\zeta_{1}^{2})(1-|\zeta_{2}|^{2})\zeta_{3},  \nonumber
\end{equation}
where $\zeta_{1} \in [0,1]$ and $\zeta_{2}, \zeta_{3}$ $\in \mathbb{\overline{D}}:=\{z \in \mathbb{C} :|z|\leq1\}$. 
\end{lemma}

\begin{lemma}\label{lem:2}
\cite{4}:
If $A$, $B$, $C \; \in \mathbb{R}$, let us consider
$$Y(A, B, C):= \max \{ |A + Bz + Cz^{2} | + 1 - |z|^2, \quad z \in \mathbb{\overline{D}} \}$$

1. If $AC \ge 0$, then
    \[Y(A, B, C)= \displaystyle\left\{\begin{array}{ll}
        |A| + |B| + |C|,& |B| \geq 2(1  - |C|), \\
        \\
        1 + |A| + \frac{B^2}{4(1 - |C|)}, & |B| < 2(1 - |C|).
    \end{array} \right.\]
2. If $AC < 0$, then
\[Y(A, B, C)=  \displaystyle\left\{\begin{array}{ll}
        1 - |A| + \frac{B^2}{4(1 - |C|)}, & -4AC(C^{-2} - 1) \le B^2 \wedge |B| < 2(1 - |C|), \\ \\
        1 + |A| + \frac{B^2}{4(1 + |C|)}, & B^2 < \min\{4(1 + |C|)^2,  -4AC(C^{-2} - 1)\}, \\ \\
        R(A, B, C), & \text{ Otherwise},
    \end{array}\right.\]
where 
\[R(A, B, C)= \displaystyle\left\{\begin{array}{ll}
        |A| + |B| - |C|, & |C|(|B| + 4|A|) \le |AB|, \\ \\
        -|A| + |B| + |C|, & |AB| \le |C|(|B| -4|A|), \\ \\
        (|C| + |A|) \sqrt{1 - \frac{B^2}{4AC}}, & \text{ Otherwise}.
    \end{array}\right. \] 
\end{lemma}

\section{\hspace{5pt} Second Hankel Determinant} \label{sec3}
By applying the definition of the Hankel determinant, we obtain the expression for second Hankel determinant given by,
\[
H_{2,1}(f) = a_1 a_3 - a_2^2.
\]
Since \( a_1 := 1 \), it follows that \( H_{2,1}(f) = a_3 - a_2^2 \). By varying the coefficients, we can compute the second Hankel determinant for different values of the coefficients. For instance, using equation \eqref{1.10}, we can derive the result for logarithmic coefficients, given by

\begin{equation}
    H_{2,1}(F_{f}/2) = \begin{vmatrix}
\gamma_1 & \gamma_{2} \\
\gamma_{2} & \gamma_{3}
\end{vmatrix} = \gamma_1 \gamma_3-\gamma_2^2 = \frac{1}{4}\left(a_2 a_4- a_3^2+\frac{1}{12}a_2^4\right).\label{3.11}
\end{equation}

It is important to note that $H_{2,1}(F_{f}/2)$ remains invariant under rotation. Since, for the function $f_\theta(z) := e^{-i \theta} f(e^{i \theta} z)$, where $f \in \mathcal{S}$ and $\theta \in \mathbb{R}$, we have the following relation:

\begin{equation}
    H_{2,1} \left( F_{f_\theta} / 2 \right) = \frac{e^{4i\theta}}{4} \left(a_2 a_4- a_3^2+\frac{1}{12}a_2^4\right)=e^{4i\theta}H_{2,1} \left( F_{f}/ 2 \right). \label{3.2}
\end{equation}

Using \eqref{1.13}, we can derive the logarithmic coefficients for the inverse functions, given by
\begin{equation}
    H_{2,1}(F_{f^{-1}}/2) = \begin{vmatrix} 
    \Gamma_1 & \Gamma_2 \\
    \Gamma_2 & \Gamma_3 
    \end{vmatrix} = \Gamma_1 \Gamma_3 - \Gamma_2^2 = \frac{1}{48} \left( 13 a_2^4 - 12 a_2^2 a_3 - 12 a_3^2 + 12 a_2 a_4 \right). \label{inv coeff}
\end{equation}

Similarly, it can be observed that $H_{2,1}(F_{f^{-1}}/2)$ is also invariant under rotation. 

\subsection{
    Sharp Bound Estimation of \texorpdfstring{$|H_{2,1}(F_{f}/2)|$}{|H2,1(Ff/2)|} 
    for \texorpdfstring{$\mathcal{S}^*_\rho$}{S* (rho)}
}

In this section, we will primarily focus on estimating the sharp bound for the Hankel determinant associated with the logarithmic coefficients for starlike functions $(|H_{2,1}(F_{f}/2)|)$ defined in the petal-shaped domain $\rho(\mathbb{D})$.

\begin{theorem} \label{thm1}
Let $f \in \mathcal{S}^*_\rho$, then
\begin{equation}
    \frac{1}{72} \le |H_{2,1}(F_{f}/2)| \leq \frac{1}{16}. \label{eq:2.1}
\end{equation} 
The above inequality is sharp.
\end{theorem}
\begin{proof} Since $f \in \mathcal{S}^*_\rho$, we can find a Swcharz function $w(z)$ such that
\begin{equation}
    \frac{zf'(z)}{f(z)}= 1+\sinh^{-1}(w(z)). \label{eq:2.3}
\end{equation}
Assume that $w(z)=(p(z)-1)/(p(z)+1)$ where, $p \in \mathcal{P}$ is given by \eqref{pp}.
Using the expansion of $f$ from \eqref{eq:1.1} and $p$ from \eqref{pp}, we get
\begin{equation}
a_{2}=\frac{1}{2}p_{1}, \; a_{3}=\frac{1}{4}p_{2}, \; a_{4}=\frac{1}{144}\left(-p_{1}^{3}-6p_{1}p_{2}+24p_{3}\right).\label{eq:petal}
\end{equation} 

Using \eqref{log coeff}, \eqref{3.11}, \eqref{eq:petal}, we obtain the following expression
\begin{eqnarray}
    \mathcal{L}&:=& \gamma_{1}\gamma_{3}-\gamma_{2}^{2}\nonumber\\
    &=&\left(\frac{1}{2} a_{2}\right)\left(\frac{1}{2}\left(a_{4}-a_{2}a_{3}+\frac{1}{3}a_{2}^{3}\right)\right)-\left(\frac{1}{2}\left(a_{3}-\frac{1}{2}a_{2}^{2}\right)\right)^{2}\nonumber\\
    &=&\frac{1}{4}\left(a_2 a_4- a_3^2+\frac{1}{12}a_2^4\right)\label{3.8}\\
    &=&\frac{1}{2304}\left(p_1^4-12p_1^2p_2-36p_2^2+48p_1p_3\right)\label{3.9}
\end{eqnarray}
By the Lemma \eqref{lem:1} and \eqref{3.9}, we get an equation in terms of the $\zeta_{i}^{'}$s, where $\zeta_{i} \in \mathbb{\overline{D}}$ for $i = 1, 2, 3$:
\begin{eqnarray}
    \mathcal{L}&=&\displaystyle{ \frac{1}{144}\left(-9\zeta_2^2+6\zeta_1^2\zeta_2^2+\zeta_1^4(-2+3\zeta_2^2)+12\zeta_1\zeta_3-12\zeta_1^3\zeta_3+12\zeta_1(-1+\zeta_1^2)\zeta_3|\zeta_2|^2\right)} \hspace{1cm}. \label{eq:3.12}
\end{eqnarray}

Since $\zeta_1 \in [0,1]$ from Lemma \eqref{lem:1}, therefore the above expression in \eqref{eq:3.12} leads to,
\[|\mathcal{L}|=  \left\{ \begin{array}{ll}
\frac{|\zeta_2|^2}{16}\leq \frac{1}{16},  &\zeta_{1}=0,\\ \\
\frac{1}{72}, &\zeta_{1}=1.
\end{array} \right.
\]

By utilising thr triangle inequality in \eqref{eq:3.12}, for $\zeta_{1} \in (0,1)$ and $|\zeta_{3}| \leq 1$, we obtain the following inequality:
\begin{eqnarray}
    |\mathcal{L}|&\leq &\displaystyle{ \frac{1}{144}|-9\zeta_2^2+6\zeta_1^2\zeta_2^2+\zeta_1^4(-2+3\zeta_2^2)+12\zeta_1\zeta_3-12\zeta_1^3\zeta_3|+ 12\zeta_1(-1+\zeta_1^2)\zeta_3|\zeta_2|^2} \nonumber\\
    &=& \frac{1}{12} \zeta_{1}(1-\zeta_{1}^{2})\;\Psi(A,B,C), \label{eq:2.5}
\end{eqnarray}
where 
\begin{eqnarray*}
        \Psi(A,B,C)&:=& |A+B\zeta_2+C\zeta_2 ^2| + 1-|\zeta_2 ^2|,
\end{eqnarray*} 

with 
\begin{eqnarray}
    A=\frac{-\zeta_{1}^{3}}{6(1-\zeta_{1}^2)},\quad
    B=0,\quad
    C=-\left(\frac{\zeta_{1}^{2}+3}{4\zeta_{1}}\right). \label{3.6}
\end{eqnarray}
\\
We now examine the following cases in the context of Lemma \ref{lem:2}, based on the various cases involving $A$, $B$, and $C$,  which are given in \eqref{3.6}.

\textbf{I.} Suppose $\zeta_{1} \in X =(0,1)$. Clearly, we can observe that $AC \ge 0$, also
\[|B|-2(1-|C|)=-2+\frac{3}{2\zeta_1}+\frac{\zeta_1}{2}>0, \;\; \zeta_1 \in X.\]
Therefore, by using Lemma \ref{lem:2}, we get, 
\[\Psi (A,B,C)=|A|+|B|+|C|.\] 
Now, by utilising above equation, \eqref{eq:2.5} can be further reduced in the following manner: 
\begin{eqnarray*}
    |\mathcal{L}| &\leq& \frac{1}{12} \zeta_{1}(1-\zeta_{1}^{2})\left(|A|+|B|+|C|\right)\\
    &=&\frac{1}{12} \zeta_{1}(1-\zeta_{1}^{2})\left(\left|\frac{-\zeta_{1}^{3}}{6(1-\zeta_{1}^2)}\right|+\left|\frac{-\zeta_{1}^{2}-3}{4\zeta_{1}}\right|\right)\\
    &=&\frac{1}{144}(9-6\zeta_1^2-\zeta_1^4)\\
    &\le& \frac{1}{16}.
\end{eqnarray*}
\\
Thus, from the above result we can see that inequality \eqref{eq:2.1} holds. To demonstrate the sharpness of \eqref{eq:2.1}, we consider the function $f_1 \in \mathcal{S}^*_\rho$ for the upper bound, given by:
\[
f_1(z) = z \exp\left( \int_0^z \frac{\sinh^{-1}(t^2)}{t} \, dt \right) = z + \frac{z^3}{2} + \frac{z^5}{8} - \frac{z^7}{144} + \ldots
\;.\]
By comparing the coefficients, we can easily find the values for $a_2$, $a_3$, and $a_4$, which upon further substitution into \eqref{3.8} yields the equality:
\[
|H_{2,1}(F_{f}/2)| = \frac{1}{16}.
\]

For the lower bound case, we consider the extremal function $f_0(z)$, defined in \eqref{1.45}, which completes the proof.
\end{proof}

\subsection{
    Sharp bound of \texorpdfstring{$|H_{2,1}(F_{f^{-1}}/2)|$}{|H2,1(Ff-1/2)|} 
    for \texorpdfstring{$\mathcal{S}^*_\rho$}{S* (rho)}
}
In this section, we will focus on estimating the sharp bound of $|H_{2,1}(F_{f^{-1}}/2)|$ within $\rho(\mathbb{D})$.

\begin{theorem}
Let $f \in \mathcal{S}^*_\rho$, then
    \begin{equation}
        \frac{1}{16} \leq |H_{2,1}(F_{f^{-1}}/2)| \leq \frac{1}{9}. \label{eq:3.1}
    \end{equation} 
The above inequality is sharp.   
\end{theorem}

\begin{proof}
Let us suppose that $f \in \mathcal{S}^*_\rho$ satisfy \eqref{eq:2.3}. Ponnusamy et al. \cite{4} examined the logarithmic coefficients of the inverses of univalent functions, defined as
\begin{equation} \label{eq:inv}
    \Gamma_1=\frac{-1}{2}a_2,\; \Gamma_2=\frac{-1}{2}a_3+\frac{3}{4}a_2^2,\; \Gamma_3=\frac{-1}{2}a_4+2a_2a_3-\frac{5}{3}a_2^3.\\
\end{equation}
By utilising \eqref{inv coeff}, \eqref{eq:inv}, and Lemma \ref{lem:1}, we express $\Gamma_{1}\Gamma_{3}-\Gamma_{2}^{2}$ in terms of $\zeta_{i}^{'}s$ by substituting $p_{i}^{'} s$ with $\zeta '_{i}s \;(i=1,2,3)$:
\begin{eqnarray} 
    \mathcal{L}&:=& \Gamma_{1}\Gamma_{3}-\Gamma_{2}^{2} \nonumber\\
    &=&\frac{1}{48}\left(13a_2^4-12a_2^2a_3-12a_3^2+12a_2a_4\right) \nonumber\\
    &=&\frac{1}{2304}\left(37p_1^4-48p_1^2p_2-36p_2^2+48p_1p_3\right) \nonumber\\
    &=&\frac{1}{144}(6\zeta_1^2(-3+\zeta_2)\zeta_2-9\zeta_2^2+\zeta_1^4(16+18\zeta_2+3\zeta_2^2)+12\zeta_1\zeta_3-12\zeta_1^3\zeta_3 \label{3.13}\\
    &\;& +12\zeta_1(-1+\zeta_1^2)\zeta_3|\zeta_2|^2) \nonumber.
\end{eqnarray}

As $|\zeta_3|\le 1$, the above inequality provides
\[|\mathcal{L}|=  \left\{ \begin{array}{ll}
\frac{|\zeta_2|^2}{16}\leq \frac{1}{16},  &\zeta_{1}=0,\\\\
\frac{1}{9}, &\zeta_{1}=1.
\end{array}\right.
\]

Let $\zeta_1 \in (0,1)$. By applying the triangle inequalitiy in \eqref{3.13} and the fact that $|\zeta_3|\le 1$, we obtain
\begin{equation}
    |\mathcal{L}|= \frac{1}{12} \zeta_{1}(1-\zeta_{1}^{2})\; \Psi(A,B,C), \label{eq:3.5}
\end{equation}
where 
\begin{equation*} 
    \Psi(A,B,C):= |A+B\zeta_2+C\zeta_2 ^2| + 1-|\zeta_2 ^2|,
\end{equation*}
with 
\begin{equation}
    A=\frac{4\zeta_1^3}{3(1-\zeta_1^2)},\; B=-\frac{3}{2}\zeta_1, \;C=-\left(\frac{3+\zeta_1^2}{4\zeta_1}\right). \label{eq:3.6}
\end{equation}

In light of Lemma \ref{lem:2} and the expressions of $A$, $B$ and $C$ obtained in \eqref{eq:3.6}, we study the following cases. Let us assume that $\zeta_{1} \in X $ where $X=(0,1)$. Clearly, we can observe that $AC<0$ for $\zeta_{1} \in X $. This case demands the following sub-cases:

\begin{enumerate}[label=(\alph*)]
    \item \label{a}  Simply for each $\zeta_1 \in X$
    \[T_1(\zeta_1):=|B|-2(1-|C|)=2\zeta_1-2+\frac{3}{2\zeta_1}>0,\]
    implies
    \[|B|>2(1-|C|).\]

    \[T_2(\zeta_1):=-4AC\left(\frac{1}{C^2}-1\right)-B^2=\frac{-\zeta_1^2(225+11\zeta_1^2)}{12(3+\zeta_1^2)}\leq 0,\]
    implies
    \[-4AC\left(\frac{1}{C^2}-1\right)\leq B^2.\]
    
     \noindent We can infer from the above computations that $T_1(X^*) \cap T_2(X^*)$ is empty. Consequently, for every $\zeta_1 \in X$, this case does not occur, as stated by Lemma \ref{lem:2}.
 \\ 
    \item  In view of $\zeta_1 \in X$, the relation $4(1 +|C|)^2$ and $- 4AC\left(\frac{1}{C^2} -1)\right)$ become
    \[T_3(\zeta_1):=4(1+|C|)^2=\frac{(3+4\zeta_1+\zeta_1^2)^2}{4\zeta_1^2}>0,\]
     \[T_4(\zeta_1):=-4AC\left(\frac{1}{C^2}-1\right)=\frac{4\zeta_1^2(-9+\zeta_1^2)}{3(3+\zeta_1^2)}<0.\]
     Therefore, $\min\{T_3(\zeta_1),T_4(\zeta_1)\}= T_4(\zeta_1)$.
    
    From case \ref{a}, we know that:
    \[-4AC\left(\frac{1}{C^2}-1\right)\leq B^2.\]
    Thus this case fails for $\zeta_1 \in X$, as stated by Lemma \ref{lem:2}.
\\
    \item For $\zeta_1 \in X$ we assume that $T_5(\zeta_1):= |A B|-|C|(|B|+4|A|)$.
    Therefore, 
    \[T_5(\zeta_1)=\frac{27+78\zeta_1^2-25\zeta_1^4}{24(-1+\zeta_1^2)}<0\]
    implies
    \[|A B|<|C|(|B|+4|A|).\]
     Thus the case fails to exist for $\zeta_1 \in X$, as stated by Lemma \ref{lem:2}.
\\
     \item For $\zeta_1 \in X$ we assume that 
     \[T_6(\zeta_1):= |A B|-|C|(|B|-4|A|) =\frac{27-114\zeta_1^2-89\zeta_1^4}{24(-1+\zeta_1^2)}\leq 0.\]
    We observe that the above equation holds true for $0 \le \zeta_1 \le \zeta' = \left(\sqrt{\frac{-57+6\sqrt{157}}{89}}\right)\approx 0.451959 \in X$.
    \noindent
    Therefore, by using Lemma \ref{lem:2}, we get, \[\Psi(A,B,C)=R(A,B,C)=-|A|+|B|+|C|.\]
    Using \eqref{eq:3.5} to proceed towards our desired result,
    \begin{eqnarray}
        |\mathcal{L}|&\leq& \frac{1}{12} \zeta_1(1-\zeta_1^2) (-|A|+|B|+|C|)\nonumber \\
         &=&\frac{1}{144}\left(9+12\zeta^2-37\zeta_1^4\right) = \phi(\zeta_1),  \label{3.16}
    \end{eqnarray}
 where
 \begin{eqnarray*}
    \phi(t)&:=&\frac{1}{144}\left(9+12t^2-37t^4\right)
 \end{eqnarray*}
Since $\phi'(t)=0$ for $t \in (0,1)$ holds true for the only possible value of $t_0=\sqrt{6/37}<\zeta_1'$, we deduce that $\phi$ is increasing in $[0,t_0]$ and decreasing in $[t_0,\zeta_1']$. Therefore, for $0<t<\zeta_1'$, we obtain
\[\phi(\zeta_1) \leq \phi(t_0) \approx 0.0692568.\]
Hence, from \eqref{3.16} we see that
\[|\mathcal{L}| \leq \phi(\zeta_1) \leq \phi(t_0) \approx 0.0692568 <1/9.\]
\\
    \item By applying Lemma \ref{lem:2} for $\zeta_1'<\zeta_1<1$, we get
    \begin{eqnarray*}
         \mathcal{L}&\leq& \frac{1}{12} \zeta_1(1-\zeta_1^2) (|C|+|A|)\sqrt{1-\frac{B^2}{4AC}}\\
          &=&\frac{1}{576} \sqrt{\frac{75 - 11 \zeta_1^2}{3 + \zeta_1^2}} \left( 9 - 6 \zeta_1^2 + 13 \zeta_1^4 \right) = \phi_2(\zeta_1)
    \end{eqnarray*}
    where
    \begin{eqnarray*}
        \phi_2(t):= \frac{1}{576} \sqrt{\frac{75 - 11 t^2}{3 + t_1^2}} \left( 9 - 6 t^2 + 13 t^4 \right)
    \end{eqnarray*}
    It is observed that $\phi_2$ is an increasing function such that $\phi_2'(t)\geq 0$ for $\zeta_1'<t<1$, therefore after a simple computation, we get 
    \[|\mathcal{L}|\leq \phi_2(\zeta_2)\leq \phi_2(\zeta_1')\approx 0.451959 <1/9.\]
\end{enumerate}
Based on the results obtained, we can see that the inequality \eqref{eq:3.1} holds true.

To demonstrate the sharpness of the bound, we can utilize the functions $f_0(z)$ and $f_{1}(z)$ for the upper and lower bounds respectively, similar to those employed in Theorem \ref{thm1}. This concludes the proof.
\end{proof}

\section{\hspace{5pt} Second Toeplitz Determinant}
By applying the definition of the Toeplitz determinant, we obtain the following expression for the second Toeplitz determinant:
\[
T_{2,1}(f) = a_1^2 - a_2^2.
\]
By varying the coefficients, the Toeplitz determinant can be computed for different sets of coefficients. For instance, by using equation \eqref{1.12}, we can determine the determinant for logarithmic coefficients:
\begin{equation} \label{4.1}
    T_{2,1}(F_{f}/2) = \begin{vmatrix}
\gamma_1 & \gamma_{2} \\
\gamma_{2} & \gamma_{1}
\end{vmatrix} = \gamma_1^2 - \gamma_2^2 = \frac{1}{16}\left(4a_2^2 - a_2^4 - 4a_3^2 + 4a_2^2a_3\right).
\end{equation}
Similarly, by applying equation \eqref{1.14}, we can compute the Toeplitz determinant for logarithmic coefficients of inverse functions:
\begin{equation} \label{4.2}
    T_{2,1}(F_{f^{-1}}/2) = \begin{vmatrix}
\Gamma_1 & \Gamma_{2} \\
\Gamma_{2} & \Gamma_{1}
\end{vmatrix} = \Gamma_1^2 - \Gamma_2^2 = \frac{1}{16}\left(-9a_2^4 + 4a_2^2 - 4a_3^2 + 12a_2^2a_3\right).
\end{equation}

Let $T_{2,1}$ be a functional and $F_{f_\theta}$, $F_{f_{{\theta}^{-1}}}$ be functions. Then, we can observe the rotational invariant characteristic of the functional $T_{2,1}$ with respect to the functions $F_{f_\theta}$ and $F_{f_{{\theta}^{-1}}}$ will be similar to $H_{2,1}(F_{f}/2)$ and $H_{2,1}(F_{f^{-1}}/2)$, as found in section \ref{sec3}.

\subsection{
    Sharp bound of \texorpdfstring{$|T_{2,1}(F_{f}/2)|$}{|T2,1(Ff/2)|} 
    for \texorpdfstring{$\mathcal{S}^*_\rho$}{S* (rho)}
}
This section explores the sharp bound estimation of the Toeplitz determinant linked to the logarithmic coefficients of starlike functions $(|T_{2,1}(F_{f}/2)|)$ defined on $\rho(\mathbb{D})$.

\begin{theorem} \label{thm3}
Let $f \in \mathcal{S}_{\rho}^{*}$, then
\begin{equation}
    \frac{1}{16} \le|T_{2,1}(F_{f}/2)| \leq \frac{1}{2}.\nonumber
\end{equation} 
The above inequality is sharp.
\end{theorem}

\begin{proof} Let us suppose that $f \in \mathcal{S}^*_\rho$ be of the form \eqref{eq:1.1}. Then there exist a Schwarz function, $w(z)=(p(z)-1)/(p(z)+1)$ such that it satisfies \eqref{eq:2.3}. From \eqref{4.1} and \eqref{eq:petal} we obtain:
\begin{eqnarray} 
    T_{2,1}(F_{f}/2)&=&\gamma_1^2-\gamma_2^2 \nonumber\\
    &=& \frac{1}{16}\left(-a_2^4+4a_2^2-4a_3^2+4a_2^2a_3\right) \label{4.5}\\
    &=&\frac{1}{256}\left(-p_1^4-4p_2^2+4p_1^2(4+p_2)\right) \label{l3}.
\end{eqnarray}
Using \eqref{l1} and \eqref{l2} from Lemma \ref{lem:1}, we derive the relation:
\[
2p_2 = p_1^2 + (4 - p_1^2)\zeta_2,
\]
where $|\zeta_2| \leq 1$. By substituting the value of $p_2$ in \eqref{l3} and assuming $\zeta_2:=\zeta$, we derive the following equation: 
\[T_{2,1}(F_{f}/2)=\frac{1}{256} \left( -p_1^4\zeta^2-16\zeta^2+16p_1^2+8p_1^2\zeta^2\right). \]
In view of triangle inequality, the above equation can be expressed as:
\[|T_{2,1}(F_{f}/2)|\le \frac{1}{256}\left( |p_1^4||\zeta^2|+ 16|\zeta^2|+16|p_1^2|+8|p_1^2||\zeta^2|\right).\]
Given the rotational invariance of $\mathcal{P}$, we restrict our analysis to non-negative values of $p_1$. Since $|p_1| \leq 2$, we have $0 \leq p_1 \leq 2$. Letting $p := |p_1|$, this implies $p \in [0, 2]$. Thus,
\begin{equation}
    |T_{2,1}(F_{f}/2)| \le \frac{1}{256} \left( p^4+24p^2+16\right).\nonumber
\end{equation}
Hence for $p \in [0,2]$,
\begin{equation}
    \frac{1}{16} \le |T_{2,1}(F_{f}/2)| \le \frac{1}{2}.\nonumber
\end{equation}
To complete the proof, it remains to verify the sharpness of the bound. Thus, in order to establish the upper bound, consider the analytic function \( f_2 \in \mathcal{S}^*_\rho \), defined by:
\[
f_2(z) = z \exp \left( \int_0^z \frac{\sinh^{-1} (\sqrt{8} i t^2)}{t} \, dt \right) = z + \sqrt{2} i z^3 - z^5 + \frac{\sqrt{2} i z^7}{9} + \ldots \;.
\]
For the lower bound, we examine the analytic function \( f_3 \in \mathcal{S}^*_\rho \), defined by:
\[
f_3(z) = z \exp \left( \int_0^z \frac{\sinh^{-1} (i t^2)}{t} \, dt \right) = z + \frac{i z^3}{2} - \frac{z^5}{8} + \frac{i z^7}{144} + \ldots \;.
\]
By comparing the coefficients, we can easily determine the values of \( a_2 \) and \( a_3 \). Substituting these into \eqref{4.5} yields the desired equality, completing the proof.
\end{proof}

\subsection{
    Sharp bound of \texorpdfstring{$|T_{2,1}(F_{f^{-1}}/2)|$}{|T2,1(Ff-1/2)|} 
    for \texorpdfstring{$\mathcal{S}^*_\rho$}{S* (rho)}
}
This section will center on estimating the sharp bound around the petal shaped domain for $|T_{2,1}(F_{f^{-1}}/2)|$, where $f \in \mathcal{S}^*_\rho$.

\begin{theorem}
Let $f \in \mathcal{S}^*_\rho$, then
\begin{equation}
    \frac{1}{16} \le |T_{2,1}(F_{f^{-1}}/2)| \leq \frac{5}{4}. \nonumber
\end{equation} 
The above inequality is sharp.
\end{theorem}

\begin{proof} Let $f \in \mathcal{S}^*_\rho$ be of the form \eqref{eq:1.1} satisfying the equation \eqref{eq:2.3}.In view of \eqref{4.2} and \eqref{eq:inv}, a simple computation shows that 
\begin{eqnarray} 
    T_{2,1}(F_{f^{-1}}/2)&=&\Gamma_1^2-\Gamma_2^2\nonumber\\
    &=& \frac{1}{16}\left(-9a_2^4+4a_2^2-4a_3^2+12a_2^2a_3\right)\nonumber\\
    &=&\frac{1}{256}\left(-9p_1^4-4p_2^2+4p_1^2(4+3p_2)\right) \nonumber.
\end{eqnarray}
Using Lemma \ref{lem:1} in above equation, we obtain 
\[T_{2,1}(F_{f^{-1}}/2)=\frac{1}{256} \left( 16p_1^2-4p_1^4+16p_1^2\zeta -4p_1^4\zeta-16\zeta^2+8p_1^2\zeta^2-p_1^4\zeta{^2}\right). \]

In view of triangle inequality, we see above equation can be expressed as:
\[|T_{2,1}(F_{f^{-1}}/2)|\le \frac{1}{256} \left(16 |p_1^2| +4|p_1^4| +16|p_1^2||\zeta| +4|p_1^4||\zeta|+ 16|\zeta^2| +8|p_1^2||\zeta^2| +|p_1^4||\zeta^2|\right).\]

Assume $p:=|p_1|$ where $p_1 \ge 0$ by analogous reasoning in Theorem \ref{thm3} and for the case $|\zeta| \le 1$, we obtain the following result by taking the help of Lemma \ref{lem:1}.
\begin{equation*}
    |T_{2,1}(F_{f^{-1}}/2)| \le \frac{1}{256} \left( 9p^4+40p^2+16\right).
\end{equation*}
Since $p \in [0,2]$. Therefore,
\begin{equation}
    \frac{1}{16} \le |T_{2,1}(F_{f^{-1}}/2)| \le \frac{5}{4}. \label{4.14}
\end{equation}

To demonstrate the sharpness of the upper bound, consider the analytic function $f_4 \in \mathcal{S}^*_\rho$, defined by:
\[
f_4(z) = z \exp \left( \int_0^z \frac{\sinh^{-1} (\sqrt{20} i t^2)}{t} \, dt \right) = z + \sqrt{5} i z^3 - \frac{5 z^5}{2} + \frac{5}{18} \sqrt{5} i z^7 + \ldots \;.
\]
For this function, a straightforward calculation yields
\[|T_{2,1}(F_{f^{-1}}/2)|=\frac{5}{4}.\]
For the sharpness of the lower bound, consider the function $f_2$, defined in Theorem \ref{thm3}. Thus, we achieve equality in \eqref{4.14} which concludes our proof.
\end{proof}

\noindent \textbf{Conflict of interest:} The authors declare no conflict of interest.\\
\noindent \textbf{Authors' Contribution:} Each author contributed equally to the research and preparation of the manuscript.


\end{document}